\documentclass[11pt]{article}

\usepackage[a4paper,margin=1.15in]{geometry}
\usepackage[T1]{fontenc}
\usepackage{lmodern}
\usepackage{microtype}
\usepackage{amsmath,amssymb,amsthm,mathtools}
\usepackage{bm}
\usepackage{enumitem}
\usepackage{tikz}
\usetikzlibrary{calc,positioning,shapes.geometric}
\usepackage[numbers,sort&compress]{natbib}
\usepackage[colorlinks=true,linkcolor=blue!50!black,citecolor=blue!50!black,urlcolor=blue!50!black]{hyperref}
\usepackage[nameinlink,capitalize,noabbrev]{cleveref}

\newtheorem{theorem}{Theorem}[section]
\newtheorem{lemma}[theorem]{Lemma}

\newtheorem{conjecture}[theorem]{Conjecture}
\theoremstyle{definition}

\theoremstyle{remark}

\newcommand{\N}{\mathbb{N}}

\usepackage{authblk}

\title{\bf The maximum number of paths of even length in a planar graph}
\author{
Zhen Liu\footnote{Email: 1552580575@qq.com},
~Chuanshu Wu\footnote{Email: cswu97@126.com (Corresponding author)}\\
{\small Center for Discrete Mathematics, Fuzhou University, Fujian, 350003, China}}
\date{}
\begin{document}
\maketitle
\begin{abstract}
For graphs $G$ and $H$, let $N(G,H)$ be the number of unlabeled, not necessarily induced copies of $H$ in $G$, and let $f(n,H)$ be the maximum of $N(G,H)$ over all $n$-vertex planar graphs $G$. Ghosh, Gy\H{o}ri, Martin, Paulos, Salia, Xiao and Zamora conjectured that, for every fixed integer $\ell\ge 2$,
\[
 f(n,P_{2\ell+1})
 =4\ell\left(\frac{n}{\ell}\right)^{\ell+1}+O(n^\ell).
\]
We prove the conjecture, including the stated error term.
Along the way, we also settle the Cox–Martin optimization conjecture.

\end{abstract}

\section{Introduction}

All graphs in this paper are finite and simple. For graphs $G$ and $H$,
let $N(G,H)$ denote the number of unlabeled, not necessarily induced
copies of $H$ in $G$, define
\[
 f(n,H)=\max\{N(G,H): |V(G)|=n \text{ and } G \text{ is planar}\}.
\]
We denote by $P_k$ and $C_k$ the path and the cycle on $k$ vertices,
respectively. Thus $P_{2\ell+1}$ is a path of even length $2\ell$.

The problem of maximizing the number of copies of a fixed graph in a
planar host graph goes back to Hakimi and Schmeichel
\cite{hakimi-schmeichel1979} and Alon and Caro
\cite{alon-caro1984}. The order of magnitude of $f(n,H)$ is now known
for every fixed planar graph $H$ by Huynh, Joret and Wood
\cite{huynh-joret-wood2022}; see also Gy\H{o}ri et
al.~\cite{gyori-etal2021} for earlier results on particular classes of graphs. Once the exponent of $n$ has been
determined, the next natural problem is to identify the leading
coefficient, or, more ambitiously, to determine the exact value. This
sharper problem is considerably more delicate.

For paths, Alon and Caro~\cite{alon-caro1984} determined $f(n,P_3)$
exactly, and Grzesik et al.~\cite{grzesik-etal2022} later determined $f(n,P_4)$ exactly. Ghosh et al.~\cite{ghosh-etal2021} proved that
$f(n,P_5)=n^3+O(n^2)$
and proposed the following conjecture for even length paths.

\begin{conjecture}[Ghosh et al.~\cite{ghosh-etal2021}]\label{conj:ghosh}
For every fixed integer $\ell\ge 2$,
\[
 f(n,P_{2\ell+1})
 =
 4\ell\left(\frac{n}{\ell}\right)^{\ell+1}
 +O(n^\ell).
\]
\end{conjecture}

The lower bound in \cref{conj:ghosh} was already supplied by Ghosh et
al.~\cite{ghosh-etal2021}.
Their construction is obtained from $C_{2\ell}$ by blowing up the vertices in one bipartition class into independent sets, with each new vertex retaining the two neighbors of the original vertex.
Choosing the $\ell$ blown-up classes to be as equal as possible gives the conjectured leading term.
It thus remains to prove the matching upper bound.

Cox and Martin~\cite{cox-martin2022} introduced a weighted optimization framework for attacking \cref{conj:ghosh}. They determined the leading
term when $\ell=3$, proving that $f(n,P_7)=\frac{4}{27}n^4+o(n^4),$
and reduced the general problem to the following weighted optimization
parameter.

Let $U$ be a finite set, and let $\mu$ be a probability measure on
the edges of the complete graph on $U$; thus
$\mu(uv)\ge 0$ and $\sum_{\{u,v\}\in\binom{U}{2}}\mu(uv)=1.$
For each $u\in U$, define its weighted degree by
$\mu(u)=\sum_{v\in U\setminus\{u\}}\mu(uv).$
For every integer $\ell\ge 2$, set
\[
 \rho(\mu;\ell)
 =
 \sum_{\substack{(u_1,\ldots,u_\ell)\in U^\ell\\
                  u_1,\ldots,u_\ell\ \mathrm{pairwise\ distinct}}}
 \mu(u_1)
 \left(\prod_{i=1}^{\ell-1}\mu(u_i u_{i+1})\right)
 \mu(u_\ell),
\]
and define
$\rho(\ell)= \sup_{U,\mu}\rho(\mu;\ell),$
where the supremum is taken over all finite sets $U$ and all probability
measures $\mu$ on $\binom{U}{2}$.
Cox and Martin proved that
\begin{equation}\label{eq:cox-martin-reduction}
 f(n,P_{2\ell+1})
 \le
 \left(\frac{\rho(\ell)}{2}+o(1)\right)n^{\ell+1}.
\end{equation}
They further proposed the following conjecture.

\begin{conjecture}[Cox--Martin~\cite{cox-martin2022}]\label{conj:cox-martin}
For every integer $\ell\ge 2$,
$\rho(\ell)=\frac{8}{\ell^\ell}.$
\end{conjecture}

Together with \eqref{eq:cox-martin-reduction},
\cref{conj:cox-martin} would give the conjectured leading coefficient
in \cref{conj:ghosh}. Cohen Antonir and Shapira
\cite{antonir-shapira2024} subsequently obtained the correct dependence
on $\ell$ up to an absolute multiplicative constant, proving that
\[
 f(n,P_{2\ell+1})
 =
 O\bigl(\ell^{-\ell}n^{\ell+1}\bigr).
\]
Their proof combined graph-theoretic reductions with arguments from
convex optimization.

For comparison, Cox and Martin also introduced an analogous
optimization problem for counting even cycles. The corresponding
conjecture has since been resolved by Lv, Gy\H{o}ri, He, Salia,
Tompkins and Zhu~\cite{lv-etal2024}. The path problem, however, remained
open: the exact leading coefficient in this family was known only for $P_5$ and $P_7$ \cite{heath-martin-wells2025}.

As a first consequence of our key weighted inequality, we determine the Cox–Martin parameter exactly.

\begin{theorem}\label{cor:rho}
For every integer $\ell\ge 2$,
$\rho(\ell)=\frac{8}{\ell^\ell}.$
\end{theorem}

Our main result confirms \cref{conj:ghosh}.

\begin{theorem}\label{thm:main}
For every fixed integer $\ell\ge 2$,
\[
 f(n,P_{2\ell+1})
 =
 4\ell\left(\frac{n}{\ell}\right)^{\ell+1}
 +O(n^\ell).
\]
\end{theorem}

\section{Preliminary lemma}\label{sec:weighted}

A weighted graph is a finite loopless simple graph whose edges have nonnegative real weights. Missing edges may equivalently be treated as edges of weight zero. For a weighted graph $F$, a vertex $r\in V(F)$ and an integer $k\ge 1$, define
\[
 R_k(F,r)
 =\sum_{\substack{r=v_0,v_1,\ldots,v_k\\v_0,\ldots,v_k\text{ pairwise distinct}}}
   \prod_{i=1}^{k}w(v_{i-1}v_i).
\]
Thus $R_k(F,r)$ is the total weight of ordered simple $k$-edge paths that start at $r$.

The following rooted-path estimate is not new. It is the case $P_{(k,0)}$ of Cohen Antonir and Shapira~\cite[Lemma~2.5]{antonir-shapira2024}, after adjoining an isolated marked vertex and scaling a probability measure to total mass $W$. We include a short direct proof for completeness.

\begin{lemma}~\cite{antonir-shapira2024}\label{lem:rooted-path}
Let $F$ be a nonnegative weighted graph with total edge weight
$W=\sum_{e\in E(F)}w(e).$
Then, for every $r\in V(F)$ and every integer $k\ge 1$,
\[
 R_k(F,r)\le \left(\frac{W}{k}\right)^k.
\]
\end{lemma}

\begin{proof}
We proceed by induction on $k$. For $k=1$, we have $R_1(F,r)=d_w(r)\le W,$
where $d_w(r)$ denotes the weighted degree of $r$.
Now suppose that $k\ge 2$ and that the assertion holds for $k-1$.
Put
$A=d_w(r)$ and $B=W-A.$
Then $B$ is the total edge weight of the weighted graph $F-r$.
Decomposing each ordered simple $k$-edge path starting at $r$
according to its first edge gives
\[
R_k(F,r)=\sum_{\substack{v\in V(F)\\rv\in E(F)}}w(rv)R_{k-1}(F-r,v).
\]
By the induction hypothesis, for every neighbor $v$ of $r$,
\[
R_{k-1}(F-r,v)\le\left(\frac{B}{k-1}\right)^{k-1}.
\]
Therefore,
\[
R_k(F,r)\le
\sum_{\substack{v\in V(F)\\rv\in E(F)}}w(rv)\left(\frac{B}{k-1}\right)^{k-1}=
A\left(\frac{B}{k-1}\right)^{k-1}
\le
\left(\frac{A+B}{k}\right)^k=
\left(\frac{W}{k}\right)^k.
\]
\end{proof}

\section{Proof of the main result}\label{sec:orientation}

We start this section with the following degree-sensitive consequence.

\begin{lemma}\label{cor:degree-sensitive}
Let $F$ be a nonnegative weighted graph of total edge weight $W$, and let $r\in V(F)$. Then, for every integer $k\ge 1$,
\[
 R_k(F,r)
 \le 2\left(\frac{W-d_w(r)/2}{k}\right)^k,
\]
where $d_w(r)$ denotes the weighted degree of $r$.
\end{lemma}

\begin{proof}
The assertion is immediate if $W=0$, so assume $W>0$. Put $D=2W-d_w(r)>0$ and define new edge weights by
\[
 y(e)=
 \begin{cases}
  w(e)/D, & e\text{ is incident with }r,\\
  2w(e)/D, & e\text{ is not incident with }r.
 \end{cases}
\]
Since the edges incident with $r$ have total weight $d_w(r)$,
\[\sum_{e\in E(F)}y(e)=\frac{d_w(r)+2(W-d_w(r))}{D}=1.
\]
Every simple path beginning at $r$ contains exactly one edge incident with $r$, namely its first edge, and hence every such $k$-edge path $P$ satisfies
\[
 \prod_{e\in E(P)}w(e)
 =\frac{D^k}{2^{k-1}}\prod_{e\in E(P)}y(e).
\]

Therefore,
\[
R_k(F,r)
=
\sum_{\substack{r=v_0,v_1,\ldots,v_k\\
v_0,\ldots,v_k\text{ pairwise distinct}}}
\prod_{i=1}^{k}w(v_{i-1}v_i)
=
\frac{D^k}{2^{k-1}}
\sum_{\substack{r=v_0,v_1,\ldots,v_k\\
v_0,\ldots,v_k\text{ pairwise distinct}}}
\prod_{i=1}^{k}y(v_{i-1}v_i).
\]
Since the total edge weight with respect to $y$ is $1$, applying
\cref{lem:rooted-path} to the weights $y$ gives
\[
 \sum_{\substack{r=v_0,v_1,\ldots,v_k\\
 v_0,\ldots,v_k\text{ pairwise distinct}}}
 \prod_{i=1}^{k}y(v_{i-1}v_i)
 \le
 \left(\frac{1}{k}\right)^k.
\]
Consequently,
\[
 R_k(F,r)
 \le
 \frac{D^k}{2^{k-1}k^k}
 =
 2\left(\frac{W-d_w(r)/2}{k}\right)^k.
\]
\end{proof}

We next prove a sharp weighted path inequality that will be the key tool for the upper bound. We now introduce the relevant weight notation.

Let $U$ be a finite set. For distinct $u,v\in U$, let $m_{uv}=m_{vu}\ge 0$, and let $p_u\ge 0$ for $u\in U$. Put
\[
 s_u=p_u+\sum_{v\in U\setminus\{u\}}m_{uv}
 \qquad\text{and}\qquad
 M=\sum_{u\in U}p_u+\sum_{\{u,v\}\in\binom{U}{2}}m_{uv}.
\]
For $\ell\ge 2$, define
\begin{equation}
 \Phi_\ell(m,p)
 =\frac12
 \sum_{\substack{(u_1,\ldots,u_\ell)\in U^\ell\\u_1,\ldots,u_\ell\text{ pairwise distinct}}}
 s_{u_1}s_{u_\ell}\prod_{i=1}^{\ell-1}m_{u_i u_{i+1}}.
 \label{eq:phi}
\end{equation}

The quantity $\Phi_\ell$ sums, over all ordered simple paths of length $\ell-1$ (i.e. with $\ell$ vertices), the product of the edge weights, multiplied by the endpoint weights $s_{u_1}$ and $s_{u_\ell}$. The factor $1/2$ is included for later convenience.

We prove a sharp upper bound for a weighted sum of ordered simple paths.
It will be applied in two ways: first, it immediately implies \cref{conj:cox-martin} for the optimization parameter $\rho(\ell)$; second, it will be used in proof of \cref{thm:main}.

\begin{lemma}\label{thm:weighted}
For every integer $\ell\ge 2$,
$\Phi_\ell(m,p)\le \frac{4}{\ell^\ell}M^{\ell+1}.$
The constant is best possible.
\end{lemma}

\begin{proof}
Construct a weighted graph $\widehat F$ whose original vertex set is $U$. Give the edge $uv$ weight $m_{uv}$ for every $\{u,v\}\in\binom{U}{2}$. For every $u\in U$, add a private leaf $u^*$ and give $uu^*$ weight $p_u$. The total edge weight of $\widehat F$ is $M$, and every original vertex $u$ has weighted degree $s_u$.

For $k\ge 1$ and $u\in U$, let
\[
 A_k(u)=
 \sum_{\substack{u=u_0,u_1,\ldots,u_k\in U\\u_0,\ldots,u_k\text{ pairwise distinct}}}
 \prod_{i=1}^{k}m_{u_{i-1}u_i}.
\]
These paths form a subset of the simple paths in $\widehat F$ that start at $u$. By \cref{cor:degree-sensitive},
\begin{equation}
 A_k(u)\le 2\left(\frac{M-s_u/2}{k}\right)^k.
 \label{eq:Ak}
\end{equation}

Using $ab\le \frac{a^2+b^2}{2}$ gives
\[
2\Phi_\ell
\le \frac12
\sum_{\substack{(u_1,\ldots,u_\ell)\\ \text{pairwise distinct}}}
\bigl(s_{u_1}^2+s_{u_\ell}^2\bigr)
\prod_{i=1}^{\ell-1}m_{u_i u_{i+1}}.
\]

Now consider the reversal map
\[
\sigma: (u_1,u_2,\dots,u_\ell) \longmapsto (u_\ell,u_{\ell-1},\dots,u_1).
\]
This map is a bijection on the set of ordered pairwise distinct $\ell$-tuples. Moreover, the edge weights are symmetric ($m_{uv}=m_{vu}$). Hence,
\[
\sum_{\substack{(u_1,\ldots,u_\ell)\\ \text{pairwise distinct}}} s_{u_1}^2 \prod_{i=1}^{\ell-1}m_{u_i u_{i+1}}
=
\sum_{\substack{(u_1,\ldots,u_\ell)\\ \text{pairwise distinct}}} s_{u_\ell}^2 \prod_{i=1}^{\ell-1}m_{u_i u_{i+1}}.
\]
Then, we have
\[
2\Phi_\ell
\le
\sum_{\substack{(u_1,\ldots,u_\ell)\\ \text{pairwise distinct}}} s_{u_1}^2 \prod_{i=1}^{\ell-1}m_{u_i u_{i+1}}.
\]
Finally, group the remaining sum by the starting vertex $u_1=u$:
\[
\sum_{\substack{(u_1,\ldots,u_\ell)\\ \text{pairwise distinct}}} s_{u_1}^2 \prod_{i=1}^{\ell-1}m_{u_i u_{i+1}}
=
\sum_{u\in U} s_u^2
\sum_{\substack{u=w_0,w_1,\dots,w_{\ell-1}\\ \text{pairwise distinct}}}
\prod_{i=1}^{\ell-1} m_{w_{i-1}w_i}.
\]
Therefore,
\[
2\Phi_\ell(m,p)
\le
\sum_{u\in U}s_u^2A_{\ell-1}(u),
\]
Applying \eqref{eq:Ak} with $k=\ell-1$ yields
\begin{equation}
 \Phi_\ell(m,p)
 \le \sum_{u\in U}s_u^2
 \left(\frac{M-s_u/2}{\ell-1}\right)^{\ell-1}.
 \label{eq:phi-intermediate}
\end{equation}

If $M=0$, the result is immediate. Assume $M>0$ and set $x_u=s_u/(2M)$. Since
\[
 \sum_{u\in U}s_u
 =\sum_{u\in U}p_u+2\sum_{\{u,v\}}m_{uv}
 =2M-\sum_{u\in U}p_u\le 2M,
\]
we have $\sum_u x_u\le 1$. Rewriting \eqref{eq:phi-intermediate},
\begin{equation}
 \Phi_\ell(m,p)
 \le \frac{4M^{\ell+1}}{(\ell-1)^{\ell-1}}
 \sum_{u\in U}x_u^2(1-x_u)^{\ell-1}.
 \label{eq:normalized}
\end{equation}
The function $x(1-x)^{\ell-1}$ on $[0,1]$ has maximum
\[
 \frac{1}{\ell}\left(\frac{\ell-1}{\ell}\right)^{\ell-1}
 =\frac{(\ell-1)^{\ell-1}}{\ell^\ell}.
\]
Consequently,
\[
 \sum_{u\in U}x_u^2(1-x_u)^{\ell-1}
 \le \frac{(\ell-1)^{\ell-1}}{\ell^\ell}\sum_{u\in U}x_u
 \le \frac{(\ell-1)^{\ell-1}}{\ell^\ell}.
\]
Substitution into \eqref{eq:normalized} proves the inequality.

For $\ell\ge 3$, equality is attained by taking $p_u=0$ and assigning weight $M/\ell$ to every edge of a cycle $C_\ell$. For $\ell=2$, equality is attained by a single edge of weight $M$.
\end{proof}

Let \( \mu \) be a probability measure on the edges of a complete graph on vertex set \( U \), and set $\bar{\mu}(u) = \sum_{v \neq u} \mu(uv).$

\begin{proof}[\textbf{Proof of Theorem~\ref{cor:rho}}]
Apply \cref{thm:weighted} with $p_u=0$, $m_{uv}=\mu(uv)$ and $M=1$. Then $s_u=\overline\mu(u)$ and
\[
 \rho(\mu;\ell)=2\Phi_\ell(m,0)\le \frac{8}{\ell^\ell}.
\]
For $\ell\ge 3$, equality is attained by the uniform probability measure on $E(C_\ell)$; for $\ell=2$, it is attained by the measure concentrated on a single edge.
\end{proof}

To prove \cref{thm:main}, we use the following classical orientation criterion. It is the upper-capacity form of Hakimi's theorem on prescribed outdegrees~\cite{hakimi1965}.

\begin{lemma}[Hakimi~\cite{hakimi1965}]\label{thm:hakimi}
Let $H$ be a finite graph and let $b:V(H)\to\N\cup\{0\}$. There exists an orientation of $H$ satisfying
$d^+(v)\le b(v)$ for $(v\in V(H))$
if and only if
\begin{equation}
 e_H(Z)\le \sum_{z\in Z}b(z)
 \qquad\text{for every }Z\subseteq V(H).
 \label{eq:hakimi}
\end{equation}
\end{lemma}

Let $G$ be a graph. Its bipartite double cover $B=B(G)$ has vertex set
$V(B)=V_L\mathbin{\cup}V_R,$
where $V_L$ and $V_R$ are two copies of $V(G)$. For every edge $uv\in E(G)$, the graph $B$ contains the two edges $u_Lv_R$ and $v_Lu_R$. We emphasize that $B(G)$ need not be planar even when $G$ is planar.

\begin{lemma}\label{lem:asymmetric-orientation}
If $G$ is planar, then $B(G)$ has an orientation satisfying
$d^+(x)\le 2$ for $x\in V_L$ and $d^+(y)\le 6$ for $y\in V_R$.
\end{lemma}

\begin{proof}
By \cref{thm:hakimi}, it is enough to verify \eqref{eq:hakimi} with capacity $2$ on $V_L$ and capacity $6$ on $V_R$. Let $A_L\subseteq V_L$ and $C_R\subseteq V_R$, and identify them with subsets $A,C\subseteq V(G)$. Put
$I=A\cap C,P=A\setminus C, Q=C\setminus A.$
Let $a=e_G(I),b=e_G(I,P),c=e_G(I,Q),$ and $d=e_G(P,Q).$
A direct count in the double cover gives $e_B(A_L,C_R)=2a+b+c+d.$
The edges counted by $b+d$ form a bipartite planar graph with parts $P$ and $I\cup Q$, so the standard planar bipartite bound gives
$b+d\le 2(|P|+|I|+|Q|).$
The edges counted by $a+c$ form a planar graph on $I\cup Q$, while the edges counted by $a$ form a planar graph on $I$. Hence
$a+c\le 3(|I|+|Q|),$ and $a\le 3|I|.$
Therefore,
\begin{align*}
 e_B(A_L,C_R)= (b+d)+(2a+c)\le 2(|P|+|I|+|Q|)+3(|I|+|Q|)+3|I|\\
 =2|P|+8|I|+5|Q|\le 2|P|+8|I|+6|Q|=2|A|+6|C|.
\end{align*}
This is exactly the capacity inequality required by \cref{thm:hakimi}.
\end{proof}

\begin{proof}[\textbf{Proof of Theorem~\ref{thm:main}}]

The lower bound follows from the construction in Ghosh et al.~\cite{ghosh-etal2021}. We now prove the upper bound.

Let $G$ be an $n$-vertex planar graph and let $B=B(G)$ be its bipartite double cover, oriented as in \cref{lem:asymmetric-orientation}. We count ordered injective copies of $P_{2\ell+1}$ in $G$. Every ordered path
$v_0v_1\cdots v_{2\ell}$
has a unique alternating lift to $B$:$(v_0)_L(v_1)_R(v_2)_L\cdots(v_{2\ell})_L.$

The orientation of $B$ induces a direction on every edge of the abstract path with vertex set $\{0,1,\ldots,2\ell\}$. Call a position a \emph{source} if all incident path edges point away from it. The set of sources forms an independent set in $P_{2\ell+1}$, so its size is at most $\ell+1$. Moreover, the unique independent set of size $\ell+1$ is $\{0,2,\ldots,2\ell\}.$
Indeed, if $0\le i_0<i_1<\cdots<i_\ell\le 2\ell$ are pairwise nonconsecutive, then $i_j\ge 2j$ for every $j$, and equality is forced throughout.

We now distinguish two cases according to the number of sources.

Case 1: at most $\ell$ sources.
Fix a set $S\subseteq \{0,1,\ldots,2\ell\}$ of source positions with $|S|=s\le \ell$.
For each position $i\notin S$, select one incident edge directed into that position.
Since every non-source position has at most two incident edges in the abstract path,
the number of ways to make these choices is at most
$2^{2\ell+1-s}.$
Fix one such choice. The selected edges form a spanning forest of the abstract path.
Every non-source has indegree one in this forest, while every source has indegree
zero; since the underlying graph is a forest, each component contains exactly one
source and is directed away from it.

Choose the images of the $s$ sources first. Their sides in the double cover are
prescribed by parity, so there are at most $n^s\le n^\ell$ choices. Then expose the
remaining vertices in root-to-leaf order. At each step, the new image must be an
outneighbour of an already exposed vertex. By \cref{lem:asymmetric-orientation},
the maximum outdegree in the oriented double cover is at most $6$. Hence, for this
fixed forest, the number of maps extending the source images is at most
$6^{2\ell+1-s}.$
Therefore, for this fixed source set $S$, the total number of compatible maps is at most
\[
2^{2\ell+1-s}\cdot n^s \cdot 6^{2\ell+1-s}
=
12^{2\ell+1-s} n^s.
\]
Since $s\le \ell$ and $2\ell+1$ is fixed, the factor $12^{2\ell+1-s}$ is bounded
by a constant depending only on $\ell$. Thus the contribution of this fixed source set $S$ is $O_\ell(n^\ell).$

We have ignored injectivity and the unselected path edges, so this is a valid upper
bound. For a fixed $s$, the number of source sets $S\subseteq \{0,1,\ldots,2\ell\}$
with $|S|=s$ is $\binom{2\ell+1}{s}$. Summing over all $s\le \ell$, the total
number of possible source sets is $\sum_{s=0}^{\ell} \binom{2\ell+1}{s},$
which is a constant depending only on $\ell$. Therefore, all paths whose source
set has size at most $\ell$ together contribute $O_\ell(n^\ell)$
ordered lifts.

Case 2: exactly $\ell+1$ sources.
By the uniqueness statement above, the only source set of size $\ell+1$ is
$S_0=\{0,2,\ldots,2\ell\}.$
In this case every edge of the lifted path is directed from $V_L$ to $V_R$. For each $x\in V_L$, \cref{lem:asymmetric-orientation} gives $|N^+(x)|\le 2$. For distinct $u,v\in V(G)$, define
\[
m_{uv}=\bigl|\{x\in V_L : N^+(x)=\{u_R,v_R\}\}\bigr|
~~\text{and}~~
p_u=\bigl|\{x\in V_L : N^+(x)=\{u_R\}\}\bigr|.
\]
Set
\[
s_u=p_u+\sum_{v\ne u}m_{uv}
~~\text{and}~~
M=\sum_u p_u+\sum_{\{u,v\}}m_{uv}.
\]
Every vertex of $V_L$ with positive outdegree contributes to exactly one summand in the definition of $M$.
Hence $M\le n.$

A lifted path has the form
$x_0,\; u_1,\; x_1,\; u_2,\; \ldots,\; u_\ell,\; x_\ell,$
where $x_i\in V_L$ for all $i$, $u_i\in V_R$, and the $u_i$ are pairwise distinct. Fix an ordered pairwise distinct sequence $(u_1,\ldots,u_\ell)$. There are at most $s_{u_1}$ choices for $x_0$, at most $s_{u_\ell}$ choices for $x_\ell$, and exactly $m_{u_i u_{i+1}}$ choices for each internal vertex $x_i$ ($1\le i\le \ell-1$). Thus the number of exceptional ordered injective paths is at most
\begin{equation}
\sum_{\substack{(u_1,\ldots,u_\ell)\\\text{pairwise distinct}}}
s_{u_1}s_{u_\ell}\prod_{i=1}^{\ell-1}m_{u_i u_{i+1}}.
\label{eq:exceptional-sum}
\end{equation}
This expression is exactly $2\Phi_\ell(m,p)$, where $\Phi_\ell$ is as defined in \eqref{eq:phi}. The product on the left may reuse a vertex of $V_L$ or violate injectivity between the two sides; both possibilities only increase the upper bound, so the inequality is valid.

Combining the estimate from Case~1 with \eqref{eq:exceptional-sum}, we obtain
\[
N(G,P_{2\ell+1})
\le \Phi_\ell(m,p)+O_\ell(n^\ell).
\]
Applying \cref{thm:weighted} and using $M\le n$, we get
\begin{align*}
N(G,P_{2\ell+1})
&\le \frac{4}{\ell^\ell}M^{\ell+1}+O_\ell(n^\ell)
\le \frac{4}{\ell^\ell}n^{\ell+1}+O_\ell(n^\ell)
=4\ell\left(\frac{n}{\ell}\right)^{\ell+1}+O_\ell(n^\ell).
\end{align*}
Since $G$ was an arbitrary $n$-vertex planar graph, this proves \cref{thm:main}.
\end{proof}

\end{document}